\documentclass[10pt]{amsart}

\usepackage{times,amsfonts,amsmath,amstext,amsbsy,amssymb,
  amsopn,amsthm,upref,eucal, amscd}
\usepackage[T1]{fontenc}

\usepackage[pdftex]{graphicx}

\newtheorem{theorem}{Theorem}[section]
\newtheorem{lemma}[theorem]{Lemma}
\newtheorem{corollary}[theorem]{Corollary}
\newtheorem{proposition}[theorem]{Proposition}

\numberwithin{equation}{section}

\theoremstyle{definition}
\newtheorem{definition}[theorem]{Definition}

\newtheorem{remark}[theorem]{Remark}

\def\Cset{\mathbb{C}}

 \def\Rset{\mathbb{R}}
 \def\Zset{\mathbb{Z}}

\def\leq{\leqslant }
\def\geq{\geqslant}

\begin{document}

\title[Regularity of $SL(2,\Rset)$-invariant measures on moduli spaces]{$SL(2,\Rset)$-invariant probability measures on the moduli spaces of translation surfaces are regular}

\author[A. Avila, C. Matheus and J.-C. Yoccoz]{Artur Avila, Carlos Matheus and Jean-Christophe Yoccoz}

\address{
CNRS UMR 7586, Institut de Math\'ematiques de Jussieu - Paris Rive Gauche,
B\^atiment Sophie Germain, Case 7012, 75205 Paris Cedex 13, France
\&
IMPA, Estrada Dona Castorina 110, 22460-320, Rio de Janeiro, Brazil
}

\email{artur@math.jussieu.fr}

\address{
Universit\'e Paris 13, Sorbonne Paris Cit\'e, CNRS (UMR 7539),
F-93430, Villetaneuse, France.
}

\email{matheus.cmss@gmail.com}

\address{
Coll\`ege de France (PSL), 3, Rue d'Ulm, 75005 Paris, France
}

\email{jean-c.yoccoz@college-de-france.fr}

\date{June 29, 2013}

\begin{abstract}
In the moduli space $\mathcal H_g$ of normalized translation surfaces of genus $g$, consider, for a small parameter $\rho >0$, those translation surfaces which have two non-parallel saddle-connections of length  $\leq \rho$. We prove that this subset of $\mathcal H_g$ has measure $o(\rho^2)$ w.r.t. any probability measure on $\mathcal H_g$ which is invariant under the natural action of $SL(2,\Rset)$. This implies that any such probability measure is regular, a property which is important in relation with the recent fundamental work of Eskin-Kontsevich-Zorich on the Lyapunov exponents of the KZ-cocycle.

\end{abstract}
\maketitle

\tableofcontents

\section{Introduction}

\subsection{Regular $SL(2,\Rset)$-invariant probability measures on $\mathcal H_g$}\label{regular}

The study of interval exchange transformations and translation flows on translation surfaces 
(such as billiards in rational polygons) was substantially advanced in the last 30 years because of its intimate relationship with the moduli spaces $\mathcal{H}_g$ of normalized (unit area) Abelian differentials on compact Riemann surfaces of genus $g\geq 1$. In fact, the moduli spaces of normalized Abelian differentials admit a natural action of $SL(2,\mathbb{R})$ and this action works as a renormalization dynamics for interval exchange transformations and translation flows. In particular, this fundamental observation was successfully applied by various authors to derive several remarkable results: 
\begin{itemize}
\item in the seminal works of H. Masur \cite{Masur1} and W. Veech \cite{Veech1} in the early 80's, the recurrence properties of the orbits of the diagonal subgroup $g_t=\textrm{diag}(e^t, e^{-t})$ of $SL(2,\mathbb{R})$ on $\mathcal{H}_g$ were a key tool in the solution of Keane's conjecture on the unique ergodicity of typical interval exchange transformations; 
\item after the works of A. Zorich \cite{Zorich1}, \cite{Zorich2} in the late 90's and G. Forni \cite{Forni} in 2001, we know that the Lyapunov exponents of the action of the diagonal subgroup $g_t=\textrm{diag}(e^t, e^{-t})$ of $SL(2,\mathbb{R})$ on $\mathcal{H}_g$ drive the deviations of ergodic averages of typical interval exchange transformations and translation flows;
\item more recently, V. Delecroix, P. Hubert and S. Leli\`evre \cite{DHL} confirmed a conjecture of the physicists J. Hardy and J. Weber on the abnormal rate of diffusion of typical trajectories in typical realizations of the Ehrenfest wind-tree model of Lorenz gases by relating this question to certain Lyapunov exponents of the diagonal subgroup $g_t=\textrm{diag}(e^t, e^{-t})$ of $SL(2,\mathbb{R})$ on $\mathcal{H}_5$.
\end{itemize}

It is worth to point out that the applications in the last two items above involved Lyapunov exponents of the diagonal subgroup $g_t=\textrm{diag}(e^t, e^{-t})$ of $SL(2,\mathbb{R})$ on the moduli spaces of normalized Abelian differentials, and this partly explains the recent literature dedicated to this subject (e.g., \cite{AV}, \cite{BM}, \cite{EKZ11}, \cite{Forni11} and \cite{MMY}). 

In this direction, the most striking recent result is arguably the formula of A. Eskin, M. Kontsevich and A. Zorich \cite{EKZ} relating sums of Lyapunov exponents to certain geometrical quantities called Siegel-Veech constants. Very roughly speaking, the derivation of this formula can be described as follows (see also the excellent  Bourbaki seminar \cite{GH} by Grivaux-Hubert  on this subject). Given $m$ an ergodic $SL(2,\mathbb{R})$-invariant probability measure on $\mathcal{H}_g$, a remarkable formula of M. Kontsevich \cite{K} and G. Forni \cite{Forni} allows to express the sum of the top $g$ Lyapunov exponents of $m$ in terms of the integral of the curvature of the determinant line bundle of the Hodge bundle over the support of $m$. The formula of Kontsevich and Forni suffices for some particular examples of 
$m$, but, in general, its range of applicability is limited because it is not easy to compute the integral of the curvature of the determinant line bundle of the Hodge bundle. For this reason, A. Eskin, M. Kontsevich and A. Zorich use an analytic version of the Riemann-Roch-Hirzebruch-Grothendieck theorem to express the integral of the curvature of the determinant of the Hodge bundle as the sum of a simple combinatorial term $\frac{1}{12}\sum\limits_{i=1}^{\sigma}\frac{k_i(k_i+2)}{k_i+1}$ depending only on the orders $k_1,\dots,k_\sigma$ of the zeroes of the Abelian differentials in the support of $m$ and a certain integral expression $I$ (that we will not define explicitly here) depending on the flat geometry of the Abelian differentials in the support of $m$. At this point, A. Eskin, M. Kontsevich and A. Zorich complete the derivation of their formula with an integration by parts argument in order to relate this integral expression $I$ mentioned above to Siegel-Veech constants associated to the problem of counting maximal (flat) cylinders of the Abelian differentials in the support of $m$.

Concerning the hypothesis for the validity of the formula of A. Eskin, M. Kontsevich and A. Zorich, as the authors point out in their article \cite{EKZ}, most of their arguments use only the $SL(2,\mathbb{R})$-invariance of the ergodic probability measure $m$. Indeed, the sole place where they need an extra assumption on $m$ is precisely in Section 9 of \cite{EKZ} for the justification of the integration by parts argument mentioned above (relating a certain integral expression $I$ to Siegel-Veech constants). 

Concretely, this extra assumption is called \emph{regularity} in \cite{EKZ} and it is defined as follows (see Subsections 1.5 and 1.6 of \cite{EKZ}). Given a normalized Abelian differential $\omega$ on a Riemann surface $M$ of genus $g\geq 1$, we think of $(M,\omega)\in\mathcal{H}_g$ as a translation surface, that is, we consider the translation (flat) structure induced by the atlas consisting of local charts obtained by taking local primitives of $\omega$ outside its zeroes. Recall that a (maximal flat) cylinder 
$C$ of $(M,\omega)$ is a maximal collection of closed parallel regular geodesics of the translation surface $(M,\omega)$. The modulus $\textrm{mod}(C)$ of a cylinder $C$ is the quotient $\textrm{mod}(C)=h(C)/w(C)$ of the height $h(C)$ of $C$ (flat distance across $C$) by the length $w(C)$ of the waist curve of $C$. In this language, we say that an ergodic $SL(2,\mathbb{R})$-invariant probability measure $m$ on $\mathcal{H}_g$ is \emph{regular} if there exists a constant $K>0$ such that 

$$\lim\limits_{\rho \to 0}\frac{m(\mathcal{H}_g(K,\rho))}{\rho^2}=0, \quad \textrm{i.e.}, \quad m(\mathcal{H}_g(K,\rho))=o(\rho^2),$$
where $\mathcal{H}_g(K,\rho)$ is the set consisting of Abelian differentials $(M,\omega)\in\mathcal{H}_g$ possessing two non-parallel cylinders $C_1$ and $C_2$ with $\textrm{mod}(C_i)\geq K$ and $w(C_i)\leq\rho$ for $i=1, 2$.

\begin{remark}\label{Siegel-Veech}
 A {\it saddle-connection} on a translation surface $(M,\omega)\in\mathcal{C}$ is a geodesic segment joining zeroes of $\omega$ without zeroes of $\omega$ in its interior.  The {\it systole} of $M$, denoted by ${\rm sys}(M)$, is the length of the shortest  saddle-connection on $M$. 
 
 Each boundary component of a  cylinder consists of unions of saddle-connections, so that $\textrm{sys}(M)\leq w(C)$ for any cylinder $C$ of  $(M,\omega)$.

It follows from the Siegel-Veech formula (cf. \cite{Veech2}, Theorem 2.2 in \cite{EM}, and Lemma 9.1 in \cite{EKZ}) that we have, for small $\rho >0$
$$ m(\{ {\rm sys}(M) \leq  \rho\}) = O(\rho^2).$$
A fortiori, those surfaces $(M,\omega) \in \mathcal H_g$  with a  cylinder $C$ satisfying $w(C) \leq  \rho$ form a set of $m$-measure $ O(\rho^2)$.
\end{remark}

\bigskip

As it is mentioned by A. Eskin, M. Kontsevich and A. Zorich right after Definition 1 in \cite{EKZ}, all \emph{known} examples of ergodic $SL(2,\mathbb{R})$-invariant probabilities $m$ are regular: for instance, the regularity of Masur-Veech (canonical) measures was shown in Theorem 10.3 in \cite{MS} and Lemma 7.1 in \cite{EMZ} (see also \cite{N} for recent related results). In particular, this led A. Eskin, M. Kontsevich and A. Zorich to conjecture (also right after Definition 1 in \cite{EKZ}) that \emph{all} ergodic $SL(2,\mathbb{R})$-invariant probabilities $m$ on moduli spaces $\mathcal{H}_g$ of normalized Abelian differentials are regular.

\subsection{The result} \label{result}

In this paper, we confirm this conjecture by showing the following slightly stronger result.  For $\rho >0$, we denote by $\mathcal H_{g,(2)}(\rho)$ the set of $M \in \mathcal H_g$ which have at least two non-parallel saddle-connections of length $\leq \rho$. 

\begin{theorem}\label{t.AMYfaible}
Let $m$ be a $SL(2,\Rset)$- invariant probability measure on $\mathcal H_g$. For small $\rho >0$, the measure of   $\mathcal H_{g,(2)}(\rho)$  satisfies
$$ m( \mathcal H_{g,(2)}(\rho)) = o(\rho^2).$$
\end{theorem}

The moduli space  $\mathcal{H}_g$ is the finite disjoint union of its strata 
$\mathcal{H}^{(1)}(k_1,\dots,k_{\sigma})$. Here,  $(k_1,\dots, k_{\sigma})$ runs amongst 
non-increasing sequences of positive integers with $\sum\limits_{i=1}^{\sigma}k_i=2g-2$, 
and $\mathcal{H}^{(1)}(k_1,\dots,k_{\sigma})$ consists of unit area translation surfaces 
$(M,\omega) \in \mathcal H_g$
such that the $1$-form $\omega$ has $\sigma$ zeroes of respective order $k_1,\dots,k_{\sigma}$.

Every stratum is invariant under the action of $SL(2,\Rset)$. 
Therefore, if $m$ is a $SL(2,\Rset)$- invariant probability measure on $\mathcal H_g$ as in the theorem,
we may consider its restriction to each  stratum of $\mathcal H_g$ (those who have positive measure) and deal separately with these restrictions.

It means that in the proof of the theorem, we may and will assume that $m$ is supported by some stratum $\mathcal{H}^{(1)}(k_1,\dots,k_{\sigma})$ of $\mathcal H_g$. Actually, we may even allow for some of the integers $k_i$ to be equal to zero, corresponding to marked points on the translation surface which are not zeroes of the $1$-form $\omega$.

A slightly more general form of our theorem is thus:

\begin{theorem}\label{t.AMY}
Let $g >0$, $\sigma >0$ and let $(k_1,\dots, k_{\sigma})$ be a non-increasing sequence of non-negative integers with $\sum\limits_{i=1}^{\sigma}k_i=2g-2$. Let $\mathcal H^{(1)} = \mathcal{H}^{(1)}(k_1,\dots,k_{\sigma})$ be the corresponding moduli space of unit area translation surfaces. For $\rho >0$, let $\mathcal H^{(1)}_{(2)} (\rho)$ be the subset consisting of translation surfaces in $\mathcal H^{(1)}$ which have at least two non-parallel saddle-connections of length $\leq \rho$. 

For any $SL(2,\Rset)$-invariant probability measure $m$ on  $\mathcal{H}^{(1)}$, one has 
$$ m( \mathcal H^{(1)}_{(2)}(\rho)) = o(\rho^2).$$

\end{theorem}

\bigskip

\subsection {Reduction to a particular case} \label{reduction} The moduli space $ \mathcal{H}^{(1)}(k_1,\dots,k_{\sigma})$ is only an orbifold, not a manifold, because some translation surfaces have non trivial automorphisms. We explain here how to bypass this (small) difficulty.

\bigskip

Denote by $\mathcal H = \mathcal{H}(k_1,\dots,k_{\sigma})$ the moduli space of translation surfaces with combinatorial data $(k_1,\dots,k_{\sigma})$ of any positive area, so that  $\mathcal H^{(1)} $ is the quotient of 
$\mathcal H$ by the action of homotheties of positive ratio.

The moduli space  $\mathcal{H}(k_1,\dots,k_{\sigma})$ is the quotient of the corresponding 
Teichm\"uller space $\mathcal{Q}(k_1,\dots,k_{\sigma})$ by the action of the mapping class group 
$MCG(g,\sigma)$. Here, a subset $\Sigma = \{ A_1,\ldots, A_{\sigma}\} $ of a (compact, connected, oriented) 
surface $M_0$ of genus $g$ is given. The Teichm\"uller space $\mathcal{Q}(k_1,\dots,k_{\sigma})$ is the 
set of of structures of translation surfaces on $M_0$ having a zero of order $k_i$ at $A_i$, 
up to homeomorphisms of $M_0$ which are isotopic to the identity ${\rm rel.} \, \Sigma$.
 The mapping class group $MCG(g,\sigma)$ is the group of isotopy classes ${\rm rel.} \, \Sigma$ 
 of homeomorphisms of $M_0$ preserving $\Sigma$.  
 \medskip
 
 The Teichm\"uller space $\mathcal{Q}(k_1,\dots,k_{\sigma})$ is a complex manifold of dimension 
 $2g + \sigma -1$ with a natural affine structure given by the period map
 \begin{eqnarray*}
  \Theta: \mathcal{Q}(k_1,\dots,k_{\sigma})& \longrightarrow & \rm {Hom}\,  (H_1(M_0,\Sigma,\Zset) , \Cset) = H^1(M_0,\Sigma,\Cset) \\
  \omega &\longmapsto & ( \gamma \longmapsto \int_{\gamma} \omega)
  \end{eqnarray*}
 which is a local homeomorphism.
 
 \medskip
 The mapping class group acts properly discontinuously on Teichm\"uller space. However, some points 
 in Teichm\"uller space may have non trivial stabilizer in $MCG(g,\sigma)$, leading to the orbifold structure for 
 the quotient space. The stabilizer of a structure of translation surface $\omega$ on $(M_0,\Sigma)$ 
 is nothing but the 
 finite group $\rm {Aut} (M_0,\omega)$ of automorphisms of this structure. The automorphism group acts freely
 on the set of vertical upwards separatrices of $\omega$. Therefore its order is bounded by the 
 number of such separatrices $\sum_i (k_i +1) = 2g -2 + \sigma$.
 
 \medskip
 
 Let $\omega_0$ be a structure of translation surface  on $(M_0,\Sigma)$; let $G$ be its automorphism group,
 viewed as a finite subgroup of $MCG(g,\sigma)$. For $\omega$ close to $\omega_0$ in Teichm\"uller space,
 the automorphism group $\rm {Aut} (M_0,\omega)$ is contained in $G$, because the action of the mapping class group is properly discontinuous. Moreover, for any $g \in G$, one has $g \in \rm {Aut} (M_0,\omega)$
 iff $\Theta(\omega)$ is a fixed point of $g$ (for the natural action of the mapping class group on 
 $H^1(M_0,\Sigma,\Cset) $). Thus, in a neighborhood of $\omega_0$, those $\omega$ which have the same 
 automorphism group than $\omega_0$ form an affine submanifold of Teichm\"uller space, and the image of this subset in the moduli space $\mathcal H$ is a manifold. The intersection of this image with $\mathcal H^{(1)}$ is also a manifold.
 
 \medskip
 
 For $i \geq 1$, denote by $\mathcal C^i$ the set of points of $ \mathcal{H}^{(1)}(k_1,\dots,k_{\sigma}) $  for which the automorphism group
 (defined up to conjugacy in the mapping class group) has order $i$. It follows from the previous discussion that 
 $\mathcal C^i$ is a manifold for all $i$, and is empty for $i > 2g-2+\sigma$. Moreover, each $\mathcal C^i$ is invariant under the action of $SL(2,\Rset)$.
 
 \medskip
 
  Let $m$ be a $SL(2,\Rset)$-invariant probability measure supported 
 on $\mathcal H^{(1)}$. To prove the property of Theorem \ref{t.AMY}, it is sufficient to prove it for the restriction of $m$
 to each of the $\mathcal C^i$ (those which have positive measure). Therefore we may and will assume below that $m$ is supported on some $\mathcal C^i$.

\subsection {Scheme of the proof  }\label{scheme}

The basic idea of the proof of Theorem \ref{t.AMY} is the following. We consider, in a level set $\{\textrm{sys}(M)=\rho_0\}$ of the systole function on $\mathcal C^i$,  the subset of translation surfaces whose shortest saddle-connections are all parallel. We will use the $SL(2,\mathbb{R})$-action to move relatively large parts of this subset to further deep sublevels  $\{\textrm{sys}(M)\leq \rho:= \rho_0\exp(-T)\}$ for $T\gg1$. The translation surfaces obtained in this way have the property that all saddle-connections not parallel to a minimizing one are much longer than the systole. A computation in appropriate pieces of  $SL(2,\mathbb{R})$-orbits, on which
the invariant probability  desintegrate as Haar measure, shows that most surfaces in the sublevel $\{\textrm{sys}(M)\leq \rho \}$ are obtained in this way. Thus we can conclude that the complement in $\mathcal C^i$, which contains the intersection  $\mathcal H_{g,(2)}(\rho) \cap \mathcal C^i$ has measure $o(\rho^2)$.

 \bigskip 

 The article is organized as follows. In  Section \ref{s.Rokhlin} , we review some material on 
 Rokhlin's desintegration theorem (\cite {R}, see also  \cite{Young}), which allows to consider separately 
 each orbit of the action of $SL(2,\Rset)$. The statement that we aim at (Proposition \ref{cond-meas}) 
 is  well-known to specialists, but we were not able to find a proper reference in the literature.
 
 \smallskip 
 
  In Section \ref{s.SL}, we discuss a couple of elementary facts about $SL(2,\mathbb{R})$ and its action on $\mathbb{R}^2$ (see Proposition \ref{decomp-grn}, Proposition \ref{changing-norms} and Proposition \ref{main-estimate}) lying at the heart of our ``orbit by orbit'' estimates. 
  
  \smallskip
  
 In Section \ref{s.conditionals}, we construct from the invariant probability $m$ a related  measure $m_0$
 which has finite total mass and  is supported on the subset $X_0^*$ of the level set 
 $\{\textrm{sys}(M)=\rho_0\}$ formed by surfaces on which all minimizing   saddle-connections are vertical.
 This measure $m_0$ enters in the formula (Corollary \ref{measure-estimate})  for the $m$-measure of 
 certain subsets of the sublevels $\{\textrm{sys}(M) \leq \rho_0 \exp(-T)\}$.
 
 \smallskip
 
  In Section \ref{s.slices}, we show that  the measure of slices of the form $\{\rho_0 \geq \textrm{sys}(M)\geq \rho_0 \exp(-\tau)\}$ when $\tau$ is small is related to the total mass of $m_0$ (Corollary \ref{derivative-of-F}).
  The slice is divided into a {\it regular part}, obtained by pushing $X_0^*$ into the slice through the action of $SL(2,\Rset)$, and a {\it singular part} , whose measure is much smaller (Proposition \ref {remaining-part}).
  
  \smallskip
  
  In Section \ref {s.tAMY}, we bring together the results of the previous sections to present the proof of the theorem.
  

\subsection {Notations}\label{notations}

We will assume some familiarity with the basic features of Abelian differentials, translation surfaces and their moduli spaces: we refer the reader to the surveys \cite{Zorich3} and \cite{Y-Pisa} for gentle introductions to the subject. We will use the following notations: 
\begin{itemize}
\item $\mathcal C$ denotes what was denoted by $\mathcal C^i$ above, namely the subset of a moduli space  $ \mathcal{H}^{(1)}(k_1,\dots,k_{\sigma}) $ formed by points whose stabilizer in the mapping class group has a given order.
\item $m$ is a $SL(2,\Rset)$-invariant probability measure supported on the manifold $\mathcal C$.
\item $(e_1,e_2)$ denotes the standard basis of $\Rset^2$ and $||\;\;||$ is the usual Euclidean norm on $\Rset^2$.
\item $R_{\theta}\in SO(2,\Rset)$ is the rotation of angle $\theta$.
\item $g_t$ is the diagonal matrix ${\rm diag} (e^t, e^{-t})$ in $SL(2,\Rset)$.
\item $n_u$ is the lower triangular matrix $ \left ( \begin{array}{cc}1&0 \\ u&1 \end{array} \right  )$.
\item $N_{a,b}$ is the upper triangular matrix   $ \left ( \begin{array}{cc}a&b \\ 0&a^{-1} \end{array} \right  )$.
\end{itemize}

 





\subsection {Two basic facts}\label{facts}

We use several times the following lemma, which is an immediate consequence of Fubini's theorem.

\begin{lemma}\label{basic}
 Let $(X,\mathcal B, m)$ be a probability space and let $G$ be a locally compact group acting measurably on $X$ by measure-preserving automorphisms, and let $\nu$ be some Borel probability measure on $G$. For any measurable subset $B \subset X$, one has 
 
 $$ m(B) = \int_X \nu(\{ g \in G , \; g.x \in B \} ) \; dm(x).$$
\end{lemma}
We will apply this taking for $\nu$ the normalized restriction of a Haar measure on $G$ to some compact subset.

\bigskip

We will use in the proof of Lemma \ref {thick-angle} the following fundamental fact on moduli space of translation surfaces (see \cite{EM}, Theorem 5.4  and also \cite{Masur2} for a much stronger result).

\begin{lemma}\label{l.EM}
For $\rho>0$, $R>0$, and any $(M,\omega) \in \mathcal{H}^{(1)}(k_1,\dots,k_{\sigma})$ with $\textrm{sys}(M)\geq\rho$, the number of holonomy vectors of saddle-connections of $(M,\omega)$ of norm $\leq R$ is bounded by a constant which depends only on $\rho$, $R$ and $k_1,\dots,k_{\sigma} $. 
\end{lemma}

\subsection*{Acknowledgements} The authors are thankful to Alex Eskin for several comments on some preliminary versions of this paper. This work was supported by the ERC Starting Grant ``Quasiperiodic'', by the Balzan project of Jacob Palis and by the French ANR grant ``GeoDyM'' (ANR-11-BS01-0004). 

\section{Conditional measures}\label{s.Rokhlin}

\subsection {The general setting}

We start by recalling the content of Rokhlin's desintegration theorem (\cite {R}, see also  \cite{Young}).

\begin{definition}
 A probability space  $(X,\mathcal B, m)$ is {\it Lebesgue}  if either it is purely atomic or 
its continuous part  is isomorphic mod.$0$ to $[0,a]$  equipped with the standard Lebesgue measure. 
Here $a \in (0,1] $ is the total mass of the continuous part of $m$.
\end{definition}

\begin{definition}
A  {\it  Polish space} is a topological space homeomorphic to a separable complete metric space. 
\end{definition}

\smallskip

Any open or closed subset of a Polish space is Polish.

\smallskip

Any Borel probability measure on a   Polish space  is Lebesgue. We generally omit in this case the reference to the $\sigma$-algebra $\mathcal B$, which is the $\sigma$-algebra generated by the Borel sets and the subsets of Borel sets of measure $0$.

\smallskip

 For a partition $\zeta$ of a set $X$,  we denote by $\zeta(x)$ the element of $\zeta$ that contains a point $x \in X$. 

\begin{definition}

Let  $(X,\mathcal B, m)$ be a Lebesgue probability space. A {\it measurable partition} of $X$ is a partition $\zeta$ of $X$ which is the limit  of a monotonous sequence 
$(\zeta_n)_{n \geq 0}$ of finite partitions by elements of $\mathcal B$. This means that, for all 
$x \in X, \;n \geq 0$, one has
$$\zeta_n(x) \in \mathcal B, \quad  \zeta_{n+1} (x) \subset \zeta_n (x), \quad \zeta(x) = \cap_{n\geq 0} \zeta_n(x) .$$  

\end{definition}
 
 As $(X,\mathcal B, m)$ is Lebesgue, the partition of $X$ by points is measurable (mod.0).

\begin{definition} Let  $\zeta$ be a measurable partition of $X$. A {\it system of conditional measures} for $(X,\mathcal B, m, \zeta)$  is a family $(m_x)_{x \in X}$ of probability measures on $(X,\mathcal B)$ satisfying the following properties:
\begin{itemize}
\item For any $x \in X$, one has $m_x( \zeta(x)) =1$.
\item For any $x,y \in X$ such that $\zeta(x) = \zeta(y)$, one has $m_x = m_y$.
\item For any $B \in \mathcal B$, the function $x \mapsto m_x(B)$ is measurable and one has
        $$m(B) = \int_X m_x(B) \; dm(x).$$
\end{itemize}
\end{definition}

\bigskip

The content of Rokhlin's theorem is that such a system of conditional measures always exists, and is essentially unique in the following sense: if $(m'_x)_{x\in X}$ is another such system, then $m_x = m'_x$ for $m$-a.a $x$.

\subsection {A special setting}

 Let $X$ be a Polish space, and let $G$ be a Lie group. 
 We will denote by $\nu$  some given  left invariant Haar  measure on $G$, 
  by $d$ some given left invariant Riemannian distance on $G$, and
  by $p_1$ the  canonical  projection from $X \times G$ onto $X$.
 We let $G$ act on the left on $X \times G$ by $g.(x,h) := (x,gh)$,
  i.e the product of the trivial action by the standard action.




\bigskip

 Let  $Z$ be a non empty  open subset of $X \times G$, and let $m$ be a Borel  probability measure on $Z$.  As an open subset of a Polish space, $Z$ is a Polish space. Thus $(Z,m)$ is a Lebesgue probability space.
 
 \medskip

 We will denote  by $m_1$  the Borel probability measure $(p_1)_* m$ on the open subset $p_1(Z)$ of $X$,
and  by  $Z_x = \{x\} \times U_x$ the fiber of $Z$ over $x$. Here $U_x$ is an open subset of $G$. The partition 
$\zeta$ of $Z$ defined by $\zeta(x,g) = Z_x$ is measurable.
 
 \medskip
 
  Let $(m_{(x,g)})_{(x,g) \in Z}$ be a system of conditional measures for $(Z,m, \zeta)$. From the second property in the definition of conditional measures,
 the measure $m_{(x,g)}$ does not depend on the second variable $g$. We will write $m_x$ instead of 
 $m_{(x,g)}$.
  For each $x \in p_1 (Z)$, $m_x$ may be seen as   a probability measure on $U_x$.


 \begin{definition}
 The Borel probability measure  $m$ on $Z$ is  {\it invariant} if, 
  for any  measurable subset $W\subset Z$, any $g \in G$ such that $g.W \subset Z$, one has $m(g.W) = m(W)$.
  \end{definition}

\begin{proposition}\label{cond-meas}
Assume  that $m$ is invariant. Then, for $m_1$-a.e $x$, the non empty open set $U_x$ has finite Haar measure
and we have 
$$m_x = \frac 1{ \nu(U_x)}  \;  \nu_{| U_x}.$$
\end{proposition} 
 \begin{proof}
 Choose a countable dense subset $D \subset G$ and denote by $\mathbb B$ the set of closed balls (for the distance $d$) in $G$ with center at a point of $D$ and positive rational radius. The proof of the proposition is an easy consequence of the following elementary lemma, whose proof will be given afterwards.

 \begin{lemma}\label{Haar-meas}
 Let $U$ be a non empty open subset of $G$ and let $\mu$ be a probability measure on $U$. 
 Assume that, for any $B \in \mathbb B$, $g \in D$ such that both $B$ and $g.B$ are contained in $U$, one has 
$\mu(g.B) = \mu(B)$. Then $U$ has finite Haar measure and $\mu = \frac 1 {\nu(U)} \nu_{|U}$. 
 \end{lemma}
 
 In view of the conclusion of the lemma, the conclusion of the proposition will follow if we know that,
  for $m_1$-a.e $x$, 
 the measure $m_x$  has the invariance property  stated in the hypothesis of the lemma. 
  \smallskip
  Let $B \in \mathbb B$, $g \in D$. The set $X(B,g)$ formed of $x \in X$ such that both $B$ and $g.B$ 
  are contained in $U_x$ is open in $X$. For any $m_1$-measurable $Y \subset X(B,g)$, one has, 
  as $m$ is invariant,
  \begin{eqnarray*}
  \int_Y m_x(B) \, dm_1(x) &=&   m(Y \times B) \\
                                              &=& m(Y \times g.B) \\
                                              &=&  \int_Y m_x(g.B) \, dm_1(x).
   \end{eqnarray*}
   It follows that $m_x(B) = m_x(g.B)$ for $m_1$-a.a $x \in X(B,g)$. Taking a countable intersection over
    the possible $B,g$ gives the assumption of the lemma so the proof of the proposition is complete.
 \end{proof}
 
 \medskip
 
 \begin{proof}[Proof of lemma \ref{Haar-meas}] We first show that $\mu$ is absolutely continuous w.r.t. the Haar measure $\nu$,
  i.e belongs to the Lebesgue class. 
  Let $B $ be any closed ball in $\mathbb B$ of small radius $r>0$ contained in $U$. There exists an integer 
  $M \geq C r^{-\dim G}$ (where the constant $C>0$ depends on $U$ but not on $B$) and elements 
  $g_1, \ldots, g_M \in D$ such that the balls $g_i .B$ are disjoint and contained in $U$. Then we have,
   from the assumption of the lemma 
  $$ M \mu(B) = \mu (\bigcup_1^M g_i.B) \leq \mu(U) =1.$$
  It follows that $\mu (B) \leq C^{-1} r^{\dim G}$ for all $B \in \mathbb B$ contained in $U$. This implies 
  that $\mu$ belongs to the Lebesgue class and that its density $\phi$ w.r.t. the Haar measure $\nu$ is bounded.
  By the Lebesgue density theorem, for $\nu$-almost all $x \in U$, one has
  $$ \phi (x) = \lim_{B \ni x } \frac {\mu(B)}{\nu(B)}$$
  where the limit is taken over balls in $\mathbb B$ containing $x$ with radii converging to $0$. If $x,x'$
   are any two points of  $U$ with this property, one can find a sequence $B_i$ of balls in $\mathbb B$ 
   with radii converging to $0$ and a sequence $g_i \in D$ such that 
   $x \in B_i \subset U$ and $x' \in g_i.B_i \subset U$ for all $i$. It follows from the invariance assumption on 
   $\mu$ and the left invariance of Haar measure that $\phi(x) = \phi(x')$. Thus the density
   $\phi$ is constant $\nu$-almost everywhere in $U$ and the proof of the lemma is complete. 
 \end{proof}

\begin{remark} \label {beware}
In the defining property of conditional measures
$$ \int_Z  f(x,g)\, dm(x,g) = \int_{p_1(Z)} \left(\int_{U_x} f(x,g) \, dm_x(g)\right) \, dm_1(x)$$
both $m_1$ and $m_x$ are {\bf probability} measures. In the context of Proposition \ref {cond-meas}, 
one may also write
$$ \int_Z  f(x,g)\, dm(x,g) = \int_{p_1(Z)} \left(\int_{U_x} f(x,g) \, d\nu(g)\right) \, (\nu(U_x))^{-1} dm_1(x).$$
However, the measure $(\nu(U_x))^{-1} dm_1(x)$ is not necessarily finite.
\end{remark}

\section {Preliminaries on $SL(2,\Rset)$}\label{s.SL}

\subsection{The decomposition $g_t R_{\theta} n_u$}\label{gRn}

Let $W$ be the set of matrices $ \left ( \begin{array}{cc}a&b \\ c&d \end{array} \right  )$ in $SL(2,\Rset)$ such that $d>0$ and $| bd | <\frac 12$.

\begin{proposition} \label {decomp-grn}
The map $(t,\theta,u) \mapsto g_t R_{\theta} n_u$ from $\Rset \times (-\frac{\pi}4, \frac{\pi}4) \times \Rset$  is a diffeomorphism onto $W$. Using this map to take  $(t,\theta,u)$ as coordinates on $W$, the restriction to $W$ of (a conveniently scaled version of ) the Haar measure is equal to $\cos 2 \theta \;  dt \; d\theta \; du$.
\end{proposition}

\begin{proof}
The vertical basis vector $e_2$  is fixed by $n_u$. For $|\theta|<\frac{\pi}4$ and $t \in \Rset$, its image $(b,d)$ under $g_t \,  R_{\theta}$ satisfy $d>0$ and $|bd|<\frac 12$. 
\smallskip

Conversely, given $ (b,d)$ satisfying these conditions, there is a  unique pair $(t,\theta) \in  \Rset \times (-\frac{\pi}4, \frac{\pi}4) $,
 depending smoothly on $(b,d)$,  such that $(b,d) = g_t R_{\theta}.e_2$. The first assertion of the proposition follows.
 
 \medskip
 
 Write the restriction to $W$ of the Haar measure on $SL(2,\Rset)$ as $\psi(t,\theta,u) dt\, d\theta \, du$, for some positive smooth function $\psi$ on $\Rset \times (-\frac{\pi}4, \frac{\pi}4) \times \Rset$. As the Haar measure is left-invariant and right-invariant, the function $\psi$ does not depend on $t$ and $u$, only on $\theta$. A small calculation, using the left-invariance under $R_{\theta}$, shows that $\psi(\theta) = \cos 2 \theta$.
\end{proof} 







\subsection {Euclidean norms along a $SL(2,\Rset)$-orbit}

The following lemma is elementary and well-known.
\begin{lemma} \label{decomp-RN}
The map $(\theta, a,b) \mapsto R_{\theta} N_{a,b} $ from $\Rset / 2\pi \Zset \times \Rset_{>0} \times \Rset$ to $SL(2,\Rset)$ (corresponding to the Iwasawa decomposition) is a diffeomorphism. The measure $d\theta \; da \; db$ 
is sent to a Haar measure by this diffeomorphism.
\end{lemma}

\begin{proposition}\label{changing-norms}
Let $v, v'$ be  vectors in $\Rset^2$ with $v \ne \pm v'$. The Haar measure of the set $E(v,v',\tau)$ consisting of the elements $\gamma \in SL(2,\Rset)$ such that
$$ || \gamma || \leq 2, \quad   \exp(-\tau) \leq || \gamma.v || \leq \exp \tau ,\quad  \exp(-\tau) \leq || \gamma.v' || \leq \exp \tau$$
is $O( \tau^{\frac 32})$ when $\tau$ is small; the implied constants are uniform when $|| v \pm v' ||$ is bounded away from $0$.
\end{proposition}

\begin{proof}
The set $E(v,v',\tau)$  is empty (for small $\tau$) unless $\frac 13 \leq || v || \leq 3, \; \frac 13 \leq || v' || \leq 3$, so we may assume that $v,v'$ are constrained by these inequalities. We may also assume that $v = (p, 0)$ with $\frac 13 \leq p \leq 3$. We write $v' = (q,r)$,  and $\gamma = R_{\theta} N_{a,b} $ as in the Lemma. The function
$$ N:  \gamma \mapsto ( || \gamma.v ||^2, || \gamma.v' ||^2) $$
does not depend on $\theta$; one has
$$ N(a,b) = (p^2 a^2, (qa +rb)^2 +  r^2a^{-2}).$$
The Jacobian matrix of $N$ is 
$$ 2  \left ( \begin{array}{cc} p^2 a&0 \\ q(qa+rb) - r^2 a^{-3}& r(qa + rb)\end{array} \right  ).$$

As $a >0$, this matrix is invertible unless $r(qa + rb) =0$, i.e  $\gamma.v$, $\gamma.v'$ are collinear or orthogonal. We conclude as follows
\begin{itemize}
\item Assume first that  $v' = \pm ( \lambda v + w)$ with $0 < \lambda  \ne 1$, $|| w || < \frac 1{100} | \lambda -1 |$. Then $\gamma.v' =  \pm ( \lambda \gamma.v + \gamma.w)$ with 
$|| \gamma.w || < \frac 2{100} | \lambda -1 |$ when $|| \gamma || \leq 2$. As $\lambda \ne 1$, we cannot have at the same time
 
$$\exp(-\tau) \leq || \gamma.v || \leq \exp \tau ,\;  \exp(-\tau) \leq || \gamma.v' || \leq \exp \tau$$

 if $\tau$ is small enough. The set $E(v,v',\tau)$ is thus empty for small $\tau$; the implied constant
depends only on $| \lambda -1 |$, i.e on $|| v \pm v' ||$.
\item Assume next that $v,v'$ are orthogonal, i.e $q =0$. Then $N(a,b) =   (p^2 a^2, r^2 (b^2 +  a^{-2}))$. The condition $\exp(-\tau) \leq || \gamma.v || \leq \exp \tau$ determines an interval for $a$ of length $O( \tau)$. For each value of $a$ in this interval, the condition $\exp(-\tau) \leq || \gamma.v' || \leq \exp \tau$ determines an interval for $b$ whose length is $O( \tau^{\frac 12})$ ( and exactly of order $\tau^{\frac 12}$ in the worst case $pr=1$). Thus, the Haar measure of $E(v,v',\tau)$ is $O( \tau^{\frac 32})$. Observe that this is still true if we relax the condition 
$|| \gamma || \leq 2$ to $|| \gamma || \leq 6$. 
\item Assume that there exists $\gamma_0 \in SL(2,\Rset)$ such that $ || \gamma_0 || \leq 3$ and 
$\gamma_0 .v,\,\gamma_0 .v'$ are orthogonal. Then the required estimate is a consequence of the 
previous case
after translating by $\gamma_0$, taking into account the  observation concluding the previous discussion.
\item Finally, assume that none of the above holds. Then the Jacobian matrix of $N$ on $|| \gamma || \leq 2$ is everywhere invertible and the norm of its inverse is uniformly bounded when
$|| v \pm v' ||$ is bounded away from $0$. In this case,  the Haar measure of $E(v,v',\tau)$ is $O(\tau^2)$, with a uniform  implied constant when $|| v \pm v' ||$ is bounded away from $0$.
\end{itemize}

\smallskip

The proposition is proved.
\end{proof}

\subsection{On the action of the diagonal subgroup}

Let $R_{\theta}$ be some given rotation. For $T>0$, consider the set 
$$J(T,\theta):= \{t \in \Rset, \; || g_t R_{\theta}. e_2 || < \exp ( -T)\}.$$

\begin{proposition} \label{main-estimate}
The set $J(T,\theta)$ is empty iff $|\sin 2 \theta | \geq \exp (-2T)$. When $|\sin 2 \theta | < \exp (-2T)$, writing  $\sin 2\theta = e^{-2T} \sin \omega$ with $\cos\omega >0$, the set 
 $J(T,\theta)$ is an open interval of length $  \frac 12 \log \frac {1+\cos\omega}  {1-\cos\omega} $.
\end{proposition}

\begin{proof}
A real number $t$ belongs to $J(T,\theta)$  iff
$$ e^{2t} \sin^2 \theta + e^{-2t}\cos^2 \theta < e^{-2T}.$$

Thus $J(T,\theta)$ is empty unless 
$$ \Delta := e^{-4T} - \sin^2 2\theta >0,$$
which gives the first assertion of the proposition.
\smallskip

When this condition holds, we write $\sin 2\theta = e^{-2T} \sin \omega$ with $\cos\omega >0$. One has $\Delta^{\frac 12} =  e^{-2T} \cos \omega$, which implies the second assertion of the proposition through the following elementary calculation. 

By performing the change of variables $x=e^{2t}$, we see that $t\in J(T,\theta)$ if and only if $x^2\sin^2\theta-e^{-2T}x+\cos^2\theta<0$, that is, $t\in J(T,\theta)$ if and only if $x=e^{2t}$ belongs to the open interval $(x_-,x_+)$ between the two roots 
$$x_{\pm}=\frac{\exp(-2T)\pm\sqrt{\Delta}}{2\sin^2\theta}=\frac{\exp(-2T)(1\pm\cos\omega)}{2\sin^2\theta}$$
In order words, $J(T,\theta)=(t_-,t_+)$ where $x_{\pm}=e^{2t_{\pm}}$. In particular, $|J(T,\theta)|=t_+-t_-=\frac{1}{2}\log\frac {1+\cos\omega}  {1-\cos\omega}$.
\end{proof}

For later use, we note that 

\begin{lemma} \label{comp-integral}
$$\int_0^{\frac {\pi}2}    \log \frac {1+\cos\omega}  {1-\cos\omega} \cos \omega \; d \omega = \pi.$$
\end{lemma}

\begin{proof}
The change of variables $u = \tan \frac {\omega} 2$ transforms the given integral into  $4 \int_0^1 \log u^{-1} \frac {1-u^2}{(1+u^2)^2} \, du$. We then have, as $\int_0^1 u^n  \log u^{-1} \, du = (n+1)^{-2}$ for $n \geq 0$
\begin{eqnarray*}
\int_0^1 \log u^{-1} \frac {1-u^2}{(1+u^2)^2} \, du &=& \int_0^1 \log u^{-1} \sum_{n\geq 0} (-1)^n (2n+1)u^{2n} \; du \\
    &=&  \sum_{n\geq 0}\frac { (-1)^n}{ 2n+1} \\
    &=& \frac {\pi} 4
\end{eqnarray*}
\end{proof}

Also for later use, let us observe that 
 
\begin{lemma}\label{l.K(omega)} Given $\omega_0>0$, there exists a constant $K=K(\omega_0)>0$ such that if
$$\exp(-2T)\sin\omega_0<|\sin 2\theta|<\exp(-2T)$$
for some $T>0$ and $\theta$, then 
$$\|g_t R_\theta e_2\|\leq K \exp(-t)$$
for all $t\in J(T,\theta)$.
\end{lemma} 

\begin{proof} Note that 
$$\|g_t R_{\theta} e_2\|^2 = e^{-2t}(\cos^2\theta + e^{4t}\sin^2\theta)\leq e^{-2t}(1 + e^{4t}\sin^2\theta)$$

In particular, from the definition of $J(T,\theta)$, we have that $e^{2t}\sin^2\theta\leq \|g_t R_{\theta} e_2\|^2<\exp(-2T)$ for $t\in J(T,\theta)$. By combining these two estimates, we deduce that
$$\|g_t R_{\theta} e_2\|^2\leq e^{-2t}(1+\exp(-2T)e^{2t})$$

On the other hand, by the elementary calculation in the end of the proof of Proposition \ref{main-estimate}, we know that, by writing $\sin2\theta=\exp(-2T)\sin\omega$,
$$e^{2t}\leq x_+:=\exp(-2T)\frac{1+\cos\omega}{2\sin^2\theta}$$
for every $t\in J(T,\theta)$.

By hypothesis, $\exp(-2T)\sin\omega_0<|\sin2\theta|<\exp(-2T)$ (i.e., $\omega_0<|\omega|<\pi/2$), so that we conclude from the previous inequality that
$$\exp(-2T)e^{2t}\leq \exp(-4T)\frac{1}{\sin^2\theta}\leq \frac{4}{\sin^2\omega_0}$$

By plugging this into our estimate of $\|g_t R_{\theta} e_2\|^2$ above, we deduce that
$$\|g_t R_{\theta} e_2\|^2\leq e^{-2t}\left(1+\frac{4}{\sin^2\omega_0}\right):=e^{-2t} K(\omega_0),$$
and thus the proof of the lemma is complete.
\end{proof}

\section{Construction of a measure  related to $m$}\label{s.conditionals}

In the next three sections, the setting is as indicated in subsection \ref{reduction} and subsection \ref {notations}:
$\mathcal C$ is a $SL(2,\Rset)$-invariant {\bf manifold} contained in moduli space, and $m$ is a $SL(2,\Rset)$-invariant probability measure supported on $\mathcal C$.

 \bigskip

Until further notice, we fix some number $\rho>0$, small enough so that the $m$-measure of $\{ M \in \mathcal C:\, {\rm sys}(M) > \rho \}$ is positive.

\smallskip
Let $X$ be the level set $ \{ M \in \mathcal C : \, {\rm sys}(M) = \rho \}$.
Let $X_0^*$ be the subset of $X$ formed of surfaces $M$ for which all non-vertical saddle-connections have
 length $> \rho$. Let $X^*:=  \bigcup_{\theta} R_{\theta}(X^*_0)$.
Observe that in this union, one has $R_{\theta}(X^*_0)= R_{\theta+\pi}(X^*_0)$ but also
 $R_{\theta_0}(X^*_0) \cap R_{\theta_1}(X^*_0) = \emptyset$ 
 for $-\frac {\pi}2 < \theta_0 < \theta_1 \leq \frac { \pi}2$. Thus we have 
$$ X^* = \bigsqcup_{-\frac {\pi}2 < \theta  \leq \frac { \pi}2}  R_{\theta}(X^*_0).$$

\medskip

The subsets  $X^*$  and $X_0^*$ are  submanifolds of $\mathcal C$ of codimension one and two respectively.



 Observe that for $|\theta| <\frac {\pi}4$, $g_t R_{\theta} e_2 $ is shorter than $e_2$ for $0<t<\log \cot |\theta|$. It follows that
 $g_t R_{\theta}(X_0^*)$ is disjoint from $X$ for $0<t < \log \cot |\theta|$, and thus the union
$$ Y^* := \bigcup_{ |\theta|< \frac {\pi}4} \bigcup_{0< t < \log \cot |\theta|} g_t R_{\theta}(X_0^*) =  \bigsqcup_{ |\theta|< \frac {\pi}4}\bigsqcup_{0< t <  \log \cot |\theta|} g_t R_{\theta}(X_0^*)$$
is a {\bf disjoint} union. 
\smallskip

In the following proposition, we use this to identify $Y^*$ with 
$ \{ (t,\theta,M) \in   \Rset \times (-\frac{\pi}4,\frac {\pi}4) \times X_0^*, \;  0\,< t < \log \cot |\theta | \}$.








\smallskip

\begin{proposition} \label{measure-on-X}
 The $m$-measure of $Y^*$ is positive . Moreover, there exists a finite measure $ m_0$ on $X_0^*$ such that the restriction of $m$ to $Y^*$ satisfies
  $$  m_{|Y^*} = dt \times  cos 2\theta \,d\theta \times   m_0 .$$
\end{proposition}

 \begin{proof}
 
  As $n_u$ fixes the vertical basis vector $e_2$, the 
 vector field $\mathfrak{n}$ which is the infinitesimal generator of the action on $\mathcal C$ of the one-parameter group $n_u$ is tangent to 
 $X_0^*$ at any point of $X_0^*$.
 
 \smallskip

 Let $M \in \mathcal C$ be such that ${\rm sys}(M) > \rho$. We claim the set of $\gamma \in SL(2,\Rset)$ 
 such that 
 $\gamma.M \in Y^*$ has non empty interior, hence positive Haar measure. It then follows from 
 Lemma \ref{basic} and the hypothesis on $\rho$  that $m(Y^*)>0$.
 
 \smallskip
 
 To prove the claim, it is sufficient to check that there exists $\gamma_0 \in SL(2,\Rset)$ such that
  $\gamma_0 . M \in X_0^* $,
 because then $g_t R_{\theta} n_u \gamma_0 .M \in Y^*$ for small $u$, $|\theta| < \frac {\pi}4$, and 
 $ 0 < t < \log \cot |\theta|$. This element $\gamma_0$ can be taken as $\gamma_0 = g_s R_{\omega}$, where 
 $\omega$ is such that the systole for $R_{\omega}.M$ is realized in the vertical direction, and 
 $s = \log \frac { {\rm sys}(M)} { \rho}$. This proves the first assertion of the proposition.

 \bigskip

  Let  $\Sigma$ be a smooth codimension-one submanifold of $X_0^* $ 
  which is transverse to the vector field $\mathfrak{n}$ (infinitesimally generating $n_u$). Taking $\Sigma$ small enough, there exists $u_0 >0$ such that 
  $n_u.\Sigma \subset X_0^*$ for $|u| < u_0$ and  $n_u.\Sigma \cap \Sigma = \emptyset$
   for $0 <u < 2 u_0$.  
  
 \medskip
 
 Then the map 
$$ \Psi_0(u,M) :=n_u.M$$
is an smooth diffeomorphism from 
$ (-u_0,u_0) \times \Sigma$
onto a subset $U$ of $X_0^*$ and the  map 
$$ \Psi(t,\theta,u,M) :=g_t R_{\theta} n_u.M$$
is an smooth diffeomorphism from 
$$\{(t,\theta,u,M) \in \Rset  \times  (-\frac{\pi}4,\frac {\pi}4)  \times (-u_0,u_0) \times \Sigma,\;  0 < t < \log \cot |\theta|\}$$
onto an open subset $V$ of $\mathcal C$. Moreover, one can choose a locally finite covering of 
 $X_0^*$ by such sets $U$. The measure $ m_0$ will be defined by its restriction to such sets.
 
 \smallskip
 
 If the $m$-measure of $V$ is $0$, the set $U$ will be disjoint from the support of $ m_0$.


\smallskip

Assume that $ m(V) >0$.
As $m$ is $SL(2,\Rset)$-invariant, the measure $\Psi^*(m_{|V})$ can be written, using Proposition \ref {cond-meas} (with $G = SL(2,\Rset)$) and Proposition \ref {decomp-grn}, as 
 
$$ dt \times  \cos 2\theta \; d\theta \times du \times \nu$$
 for some finite  measure $\nu$ on $\Sigma$ (cf. Remark \ref {beware}: the open subset $U_x$ of 
 Proposition \ref {cond-meas} is here independent of $x$).

  One thus defines 
 
 $$   m_{0|U} =( \Psi_0)_* (du \times \nu).$$
 
 Then  $ m_0$ satisfies the required conditions. 
\end{proof}

In the next corollary, $J(T, \theta)$ is the interval that has been defined in Proposition \ref{main-estimate}.

\begin{corollary} \label{measure-estimate}
Let $B$ be a Borel subset of $X_0^*$, $\omega_0 \in (0,\frac{\pi}2]$, $T>0$. Let $Y(T,\omega_0, B)$ 
be the set of surfaces $M' = g_t R_{\theta} M$ in $Y^*$ such that $M \in B$, 
$| \sin 2 \theta | < \exp(-2T) \sin \omega_0$ and $ t \in J(T, \theta)$. One has 
$$ m(Y(T,\omega_0, B)) =  \frac 14  \exp(-2T) \, m_0 (B)  \int_{-\omega_0}^{\omega_0}    \log \frac {1+\cos\omega}  {1-\cos\omega} \cos \omega \; d\omega $$

\end{corollary}

\begin{proof}
Write  $\sin 2\theta = e^{-2T} \sin \omega$ as in Proposition \ref{main-estimate} , so that the condition 
$| \sin 2 \theta | < \exp(-2T) \sin \omega_0$ is equivalent to $|\omega| < \omega_0$. 
As the length of $J(T, \theta)$ is 
  $$ \frac 12 \log \frac {1+\cos\omega}  {1-\cos\omega} $$
we obtain  from Proposition \ref{measure-on-X}
 \begin{eqnarray*}
m(Y(T,\omega_0, B)) & = & m_0 (B) \int_{|\sin 2 \theta| < \exp(-2T) \sin \omega_0 }  |J(T,\theta)| \cos 2 \theta d\theta \\
                                  & = & \frac 14  \exp(-2T)  m_0 (B)  \int_{-\omega_0}^{\omega_0}    \log \frac {1+\cos\omega}  {1-\cos\omega} \cos \omega \; d\omega. \\
   \end{eqnarray*}
  \end{proof}

\section{Measure of the  slice $\{\rho \geq {\rm sys}(M) \geq \rho \exp (-\tau)\}$}\label{s.slices}

For $\rho >0$, we denote by $F(\rho)$ the $m$-measure of the set $\{ sys(M) \leq \rho\} $. This is a non-decreasing function of $\rho$. 

\smallskip

For any $\rho >0$, the level set $\{ sys(M) = \rho\} $ has $m$-measure $0$, 
as may be seen for instance by an elementary application of Lemma \ref{basic}. 
It follows that the function $F$ is continuous. We will prove in this section (see Corollary \ref{derivative-of-F}) 
that, for any $\rho$ such that $F(\rho)<1$, the function $F$ has at $\rho$ a positive left-derivative $F'(\rho)$
which is equal to $\pi \rho^{-1}  m_0(X_0^*)$, where $X_0^*$, $m_0$ are as 
Proposition \ref {measure-on-X}.

\medskip

In this section, $\rho$ is as above a positive number such that $F(\rho) <1$.

\subsection {The regular part of the slice}

For $M \in X^*$ and $t\geq 0$, we define 
$$ \Phi_t(M):= R_{\theta} g_t R_{-\theta} (M), \quad {\rm when} \quad M \in R_{\theta}(X^*_0).$$
Observe that the two possible choices for $\theta$ give the same result as $R_{\pi}$ is the non trivial element of the center of $SL(2,\Rset)$.

\smallskip

The systole of $ \Phi_t(M)$ is equal to $\rho \; \exp (-t)$. For $ M \in R_{\theta}(X^*_0)$, all saddle connections of $ \Phi_t(M)$ with minimal length have the same angle $\theta$ with the vertical direction. It follows that the $\Phi_t$ are injective, and that $\Phi_t(X^*) \cap \Phi_{t'} (X^*) = \emptyset$ for $t \ne t'$.

\smallskip

\begin{definition}
For $\tau>0$, the set  $\bigsqcup_{0 \leq t \leq \tau} \Phi_t (X^*)$ is called the {\it regular part} of the slice  
$$ S(\tau):= \{\rho \geq {\rm sys}(M) \geq \rho \exp (-\tau)\}.$$
 Its complement (in $S(\tau)$) is called the {\it singular part} of the slice $S(\tau)$.
\end{definition}

Next, we define, for a Borel subset $B$ of $X^*$ and $\tau>0$

$$\tilde m_{\tau}(B):= \frac 2{1- \exp(-2\tau)} \; m( \bigsqcup_{0 \leq t \leq \tau} \Phi_t (B)).$$

\begin{proposition}\label{regular-part}
The measure $\tilde m_{\tau}$ is independent of $\tau$, and  is equal to the product $d\theta \times  m_{0} $.
\end{proposition}

\begin{proof}
As  the measure $\tilde m_{\tau}$ on $X^*$  is invariant under the action on $X^*$ of the group of rotations, 
it follows from Proposition \ref {cond-meas} that the measure $\tilde m_{\tau}$ can be written as   
$ d\theta  \times    m_{\tau} $, for some finite measure $  m_{\tau}$ on $X_0^*$ . 
We will show that  $  m_{\tau}=   m_{0}$.

\smallskip

Let $\Sigma$ and $u_0$ be as in the proof of Proposition \ref {measure-on-X}. 
For $(u,M) \in   (-u_0,u_0) \times \Sigma$, any $\theta$ and any $t \geq 0$,  we have
$$ \Phi_t( R_{\theta} n_u.M) =  R_{\theta} g_t n_u.M.$$

Observe that for $t \geq 0, \, \theta \in  (-\frac{\pi}4, \frac{\pi}4) $, the element  $R_{\theta} g_t $
 belongs to the set $W$ of subsection \ref{gRn}, 
so that we can write, according to Proposition \ref {decomp-grn}
$$ R_{\theta} g_t = g_{T(t,\theta)}  R_{\Theta(t,\theta)} n_{U(t,\theta)} $$
for some smooth functions $T,\Theta,U$. For $\theta$ close to $0$, we have
$$ T(t, \theta) =  t + O(\theta), \quad U(t,\theta) = O( \theta),   \quad \Theta(t,\theta) = e^{-2t} \theta + O(\theta^2) .$$

We will use this local information on $T(t,\theta)$, $\Theta(t,\theta)$ and $U(t,\theta)$ combined with our expression for the measure $m|_{Y^*}$ in $g_T R_{\Theta} n_{U}$-coordinates (cf. Proposition \ref{decomp-grn} and the proof of Proposition \ref{measure-on-X}) to compute $\tilde m_{\tau}$ as follows.

Let $B_0 = (u_1,u_2) \times B$ be an elementary Borel subset of $(-u_0,u_0) \times \Sigma \subset X_0^*$, 
where $B$ is a Borel subset of $\Sigma$. For small $\theta >0$, we have, 
in view of the formula for Haar measure in   Proposition \ref {decomp-grn}

\begin{eqnarray*}
\tilde m_{\tau} ( [0, \theta]  \times B_0) & =&  \frac 2{1- \exp(-2\tau)} \int_{B} \int_{u_1}^{u_2} \int_0^{\tau}   \,e^{-2t} \, \theta \, dt\, du \, d\nu   + O(\theta^2)\\
& = &  m_0(B_0)\,  \theta + O(\theta^2).
\end{eqnarray*} 

This proves that $  m_{\tau}=   m_{0}$.
\end{proof}

\subsection{Estimate for the  singular part of the slice} 

 Let $Z(\tau)$ be the subset of  $S(\tau)$ consisting of surfaces $M$ having a saddle-connection of length 
 $\leq \rho \exp \tau$ which is not   parallel to a minimizing one. 
  Clearly the singular part of the slice $S(\tau)$  is contained in $Z(\tau)$.

\begin{proposition} \label{remaining-part}
For small $\tau>0$ , the $m$-measure of $Z(\tau)$ (hence also the $m$-measure of the singular part of the slice  $S(\tau)$) is $o(\tau)$.
\end{proposition}

From Propositions \ref{regular-part} and \ref{remaining-part}, we get the

\begin{corollary}\label{derivative-of-F}
One has 
$$ \lim_{\tau \rightarrow 0} \frac 1{\tau} m(\{\rho \geq {\rm sys}(M) \geq \rho \exp (-\tau)\}) = \pi  m_0 (X_0^*). $$
\end{corollary}

\bigskip

\begin{proof}[Proof of Proposition \ref{remaining-part}]
For $M \in S(\tau)$, let us denote by $\hat \theta (M)$ the smallest non-zero angle between
 two saddle-connections with length $\leq 3 \rho$ (If all connections with length
$\leq 3 \rho$  are parallel, we define  $\hat \theta (M) = \frac {\pi}2$). 
The required estimate for the measure of $Z(\tau)$ follows from two lemmas:

\begin{lemma}\label{thin-angle}
For any $\eta >0$, there exists $\hat \theta_0$ such that, for any $\tau >0$ small enough, one has
$$ m(\{ M \in S(\tau) : \; \hat \theta(M) < \hat \theta_0 \} ) < \eta \tau.$$ 
\end{lemma}

\begin{lemma}\label{thick-angle}
For any $\hat \theta_0$, one has 
$$ m(\{ M \in Z(\tau) : \; \hat \theta(M) \geq  \hat \theta_0 \} ) =  O(\tau^{\frac 32}),$$
where the implied constant depends on $\hat \theta_0$,$\rho$ and $g$.
\end{lemma}

\begin{proof}[Proof of Lemma \ref{thin-angle}] 
Denote by $S_1(\tau)$  the subset of $S(\tau)$ consisting of translation  surfaces for which 
 there exists a  length-minimizing saddle-connection  whose direction form  an angle $\leq \frac {\pi}6$ with the vertical direction. 
 
 \bigskip


It follows from Lemma \ref{basic} (with $G=SO(2,\Rset)$ equipped with the Haar measure) 
that for any $SO(2,\Rset)$-invariant subset 
$\mathcal S \subset S(\tau)$,  we  have
$$ m(\mathcal S) \leq 3\; m ( S_1(\tau) \cap \mathcal S).$$
We will apply this relation with
$$ \mathcal S = \{ M \in S(\tau), \; \hat \theta(M) < \hat \theta_0 \},$$
for some appropriate $\hat \theta_0$.

\smallskip

 For $ M \in S_1(\tau)$, and any positive integer $j $ such that 
 $$ \exp (3j\tau) <  \cot \frac{\pi}6 = \sqrt 3$$
 the systole of $ g_{3j\tau}.M$ is
 $<  \rho \exp (-\tau)$. Therefore the images of $S_1(\tau)$ under the elements $ g_{3j\tau}$, for $0\leq j < \frac  {\log 3}  {6 \tau} -1 $, are disjoint.
 Observe also that for such $M$, $j$, the surface $M':=  g_{3j\tau}.M$ has a systole in $(\frac {\rho}2, \sqrt 3 \rho)$ and has two saddle-connections of length at most 
$3 \sqrt 3 \rho$ with a non-zero angle $ \leq A \hat \theta (M)$, for some absolute constant $A>1$. 

\medskip

Let $\eta >0$. Let $\hat \theta_1 >0$ be small enough in order that the $m$-measure of the set of surfaces 
$M'$ with systole in $(\frac {\rho}2, \sqrt 3 \rho)$, having
two non-parallel saddle-connections of length $\leq 3\sqrt 3 \rho$ and angle
 $< \hat  \theta_1$,  is $< \frac {\eta}{18}$.
Choosing $\hat \theta_0 := A^{-1}\hat  \theta_1$, as the number of values of $j$  with 
$0\leq j < \frac  {\log 3}  {6\tau} -1 $  is $\geq \frac 1{6\tau}$ for $\tau$ small enough, we obtain
$$ m(\{ M \in S(\tau), \; \hat \theta(M) < \hat \theta_0 \} \cap \mathcal S_1) < 6\tau .\frac {\eta}{18}, $$
$$ m(\{ M \in S(\tau), \; \hat \theta(M) < \hat \theta_0 \}) < \eta \tau.$$
\end{proof}


\begin {proof}[Proof of Lemma \ref{thick-angle}] 
We will apply Lemma \ref{basic} with 

$$B: = \{ M \in Z(\tau), \; \hat \theta(M) \geq  \hat \theta_0 \}, $$ 
$G = SL(2,\Rset)$, and $\nu$ the normalized restriction of a Haar measure to
$$K:= \{ \gamma \in SL(2,\Rset), \; || \gamma || \leq 2 \},$$ 
where $|| \;\;||$ is the Euclidean operator norm. We therefore have to estimate the relative measure in $K$ 
of the sets $\{ \gamma \in K, \; \gamma.M \in B \}$.

\smallskip

If $K.M$ does not intersect $B$, this measure is $0$.

\smallskip

Let $M \in \mathcal C,\,  \gamma \in K$such that  $\gamma.M \in B$ . As 
$||\gamma|| = ||\gamma^{-1}|| \leq 2$, we have
$$ \frac 12 \rho \exp(-\tau) \leq {\rm sys} \, (M) \leq 2 \rho.$$

As $B \subset Z(\tau)$, there exist non colinear holonomy vectors $v,v'$ of saddle connections of $M$
such that 
$$ \rho \exp(-\tau) \leq ||\gamma.v|| \leq \rho \exp (\tau) ,  \quad \rho \exp(-\tau) \leq ||\gamma.v'|| \leq \rho \exp (\tau) .$$
This means that $\gamma$ belongs to the set $E\left(\rho^{-1}v, \rho^{-1} v',\tau\right)$ of Proposition 
\ref{changing-norms}.

As $||\gamma|| = ||\gamma^{-1}|| \leq 2$, we must have (for small $\tau$) $||v|| \leq 3 \rho, \,||v'|| \leq 3 \rho$.

Moreover, the angle between the directions of $v, \,v'$ is $\geq A^{-1} \hat \theta_0$ for some appropriate
 absolute constant $A>1$: otherwise, we would have $\hat \theta (\gamma.M) < \hat \theta_0$, 
 in contradiction to $\gamma.M \in B$. This imply that $|| \rho^{-1} (v  \pm v')||$ is bounded from below 
 by a constant $c$ depending only on  $\hat \theta_0$.

\smallskip

Let $v_1, \dots, v_N$ be the holonomy vectors of saddle-connections of $M$ of lengths  $\leq 3\rho$. 

 We have shown that 
$$\{\gamma\in K : \, \gamma.M\in B\}\subset \bigcup E\left(\rho^{-1} v_i, \rho^{-1} v_j,\tau\right)$$ 
where the union is taken over indices $i,j$ such that $|| \rho^{-1} (v_i  \pm v_j)|| \geq c$. In view of Proposition 
\ref{changing-norms}, we obtain

$$\frac{\nu(\{\gamma\in K: \gamma.M\in B\})}{\nu(K)}\leq N^2\cdot O_{\hat\theta_0}(\tau^{3/2})$$

By Lemma \ref{l.EM} in Subsection \ref {facts}, the integer $N$ has an upper bound 
depending only on $\rho$ and $g$. 

The statement in Lemma \ref {thick-angle}  follows from this estimate, plugged into Lemma \ref{basic}. 
\end{proof}

This completes the proof of Proposition \ref{remaining-part}.
\end{proof}

\section{Proof of Theorem \ref{t.AMY}}\label{s.tAMY}

We recall that $F(\rho)$ stands for the  $m$-measure of the set $\{ sys(M) \leq \rho\} $.
 It is a continuous non-decreasing function of $\rho >0 $. By Corollary \ref{derivative-of-F}, 
 at any $\rho >0$ such that $F(\rho)<1$, the function $F$ has a {\bf positive}  left-derivative 
$F'(\rho)$ which is  equal to $\pi \rho^{-1}  m_0(X_0^*)$ (where $X_0^*$, $ m_0$ are as in Proposition \ref{measure-on-X}) .
 
 \medskip
 
 \begin{proposition}\label{Estimate-of-F}
 For any $\rho >0$ with $F(\rho)<1$, and any $T >0$, we have
 $$ F( \rho \exp(-T) ) \geq  \frac 12  \exp(-2T) \rho F'(\rho) .$$
 \end{proposition} 
\begin{proof}
We apply Corollary \ref{measure-estimate} with $B  = X_0^*$, and $\omega_0 = \frac {\pi}2$. 
From this corollary and Lemma \ref{comp-integral}, we obtain 
  \begin{eqnarray*}
m(Y(T,\frac{\pi}2, X_0^*)) &=&  \frac 14  \exp(-2T)  m_0 (X_0^*)  \int_{-\frac{\pi}2}^{\frac{\pi}2}    \log \frac {1+\cos\omega}  {1-\cos\omega} \cos \omega \; d\omega \\
                                             &=&  \frac {\pi}2  \exp(-2T)  m_0 (X_0^*)\\
                                             &=&  \frac 12  \exp(-2T) \rho F'(\rho).
 \end{eqnarray*}

To get the estimate of the proposition, it is now sufficient to recall that, by Proposition \ref{main-estimate},
any surface in  $Y(T,\frac{\pi}2, X_0^*)$ has a systole $\leq  \rho \exp(-T)$.
\end{proof}

\begin{corollary}\label{abs-cont}
The function $F$ is absolutely continuous, and its left-derivative $F'(\rho)$ satisfies
 $$F'(\rho) = O(\rho).$$
Moreover, defining
  $$ c(m) := \frac 12 \sup_ {F(\rho) < 1} \frac {F'(\rho)} {\rho}, $$
one has
$$ \lim_{\rho \rightarrow 0} \frac{F(\rho)}{\rho^2} = c(m) $$
and also
$$ c(m) := \frac12 \limsup_ {\rho \rightarrow 0 } \frac {F'(\rho)} {\rho}.$$
\end{corollary}

\begin{proof}
It follows from Proposition \ref{Estimate-of-F} that we have $F'(\rho) \leq 2 \rho^{-1}F(\rho)$ for $F(\rho)<1$.
Recalling (cf. Remark  \ref{Siegel-Veech}) that $F(\rho) = O(\rho^2)$, we obtain that $F'(\rho) = O(\rho)$.
In particular, $F'$ is bounded.

\smallskip

If $f$ is a continuous function on some interval  $[\rho_0, \rho_1]$, having at each point a left-derivative
bounded by $C$, then one has 
$$ |f(\rho_1) - f(\rho_0)| \leq C(\rho_1 - \rho_0)$$
and thus $f$ is absolutely continuous. Indeed, it is sufficient to prove this inequality for any $C' >C$. Given $C'>C$, let 
$$I(C'):=\{\rho\in[\rho_0,\rho_1]: |f(\rho_1)-f(\rho)|\leq C'(\rho_1-\rho)\}.$$
Note that $\rho_1\in I(C')$ (and hence $I(C')$ is not empty), $I(C')$ is closed (because $f$ is continuous) and $I(C')$ is ``open to the left'' in the sense that, for each $\rho_*\in I(C')$, $\rho_0<\rho_*\leq\rho_1$, there exists $\delta=\delta(\rho_*)>0$ such that $[\rho_*-\delta,\rho_*]\subset I(C')$ (because of the bound $|f'(\rho_*)|\leq C<C'$ on the left-derivative). From these properties, we deduce that $I(C')=[\rho_0,\rho_1]$ and therefore $f(\rho_1) - f(\rho_0)| \leq C(\rho_1 - \rho_0)$.
 
\smallskip

Applying this with $f=F$, we conclude that $F$ is absolutely continuous.

\smallskip

 Defining $c(m)$ as in the corollary, we get from Proposition \ref{Estimate-of-F} that
$$ \liminf_{\rho \rightarrow 0} \frac{F(\rho)}{\rho^2} \geq c(m) .$$

On the other hand, since $F$ is absolutely continuous, it has a derivative at almost every point and it is the integral of its almost everywhere derivative. In particular, in our setting, it follows that $F$ is the integral of its left-derivative $F'$, so that
$$  \limsup_{\rho \rightarrow 0} \frac{F(\rho)}{\rho^2} \leq c(m).$$

This proves that  $ \lim_{\rho \rightarrow 0}  \rho^{-2}F(\rho) = c(m) $. The last assertion of the corollary is then obvious. 
\end{proof}

\bigskip

\begin{proof}[Proof of  Theorem \ref{t.AMY}] 
We have to prove that the $m$-measure of the set of $M \in \mathcal C$ which have at least two non-parallel saddle-connections 
of length $\leq \rho$ is $o(\rho^2)$ when $\rho$ is small. Taking into account Remark \ref{Siegel-Veech}, it is sufficient to prove that, for any $A>1$, the set
 $\mathcal C(A,\rho)$ formed by the points $M \in \mathcal C$ with ${\rm sys}(M) \leq \rho$ having another saddle-connection of length $\leq A \cdot {\rm sys}(M)$, 
 non parallel to the minimizing one, has $m$-measure $o(\rho^2)$ when $\rho$ is small.
 
 \bigskip
 
 Fix some $A >1$, and some $\eta >0$.
 
 
 We will prove that 
 
 \begin{equation}\label{e.A}
  m( \mathcal C(A,\rho )) < \eta \rho^2,
  \end{equation}
 
 when $\rho$ is sufficiently small.
 
 \medskip
 
 We first choose $\rho_0$ sufficiently small to satisfy  $F(\rho_0)<1$ and

 \begin{equation}\label{e.F'}
 \rho_0^{-1} F'(\rho_0) >  2 c(m) - \frac {\eta}2;
 \end{equation}
 \begin{equation} \label{e.F}
 \rho^{-2} F(\rho) < (c(m) + \frac {\eta}4) \quad {\rm  for} \; 0<\rho <\rho_0.
  \end{equation}
 
 We use $\rho_0$ as the level of the systole function at which are defined $X_0^*$, $ m_0$ in  Section
 \ref {s.conditionals}.
 
 \medskip
 
Let $\omega_0  = \omega_0(\eta) >0$ be small enough to have 
\begin{equation}\label{e.omega}
       m_0 (X_0^*)  \int_{-\omega_0}^{\omega_0}    \log \frac {1+\cos\omega}  {1-\cos\omega} \cos \omega \; d\omega < \eta \rho_0^2. 
  \end{equation}      



Next, consider $T>0$ and  $\theta$ such that $  \exp (-2T)  \sin \omega_0  <  |\sin 2\theta| < \exp (-2T) $. By Lemma \ref{l.K(omega)}, there exists 
$K = K(\omega_0)$, {\bf independent of $T$},
such that, for all $t \in J(T, \theta)$, the norm of the image of the vertical basis vector $e_2$ under $g_t R_{\theta}$ is $\leq K \exp (-t)$.

\smallskip

Thus, if $v$ is any vector with $||v|| \geq AK$, we will have, for such $\theta$ and $t$

\begin{equation}\label{e.B}
 ||g_t R_{\theta} v || \geq AK \exp (-t) \geq A   ||g_t R_{\theta} e_2||. 
 \end{equation}

\medskip

Let $M \in X_0^*$. Denote by $\bar \theta(M)$ the minimal  angle between a length-minimizing saddle-connection and another non-parallel saddle connection of length 
$\leq AK \rho_0$.
If no such short non-parallel saddle connection exists, set $\bar \theta(M) = \frac {\pi}2$.

\smallskip

 As $\bar \theta$ is everywhere positive on $X_0^*$, there exists $\bar \theta_0$ such that the set 
 $B(\bar \theta_0):= \{ M \in X_0^*, \; \bar \theta (M) < \bar \theta_0 \}$ satisfies

\begin{equation} \label {e.theta}
 \frac {\pi}2   m_0 (B(\bar \theta_0)  )  < \frac {\eta}4\rho_0^2 .
 \end{equation}

On the other hand, if $T$ is sufficiently large (say $T \geq T_0 = T_0(A, \bar \theta_0)$), one has, for 
$    |\sin 2\theta| < \exp (-2T) $ ,
 $ \frac {\pi}2 > |\theta'| > \bar \theta_0$,  $t  \in J(T,\theta)$
\begin{equation}\label{e.C}
 || g_t R_{\theta + \theta'} e_2 || \geq A || g_t R_{\theta } e_2 || .
 \end{equation}
In fact, recall that $t\in J(T,\theta)$ if and only if $\|g_t R_\theta e_2\|^2<\exp(-2T)$ and $J(T,\theta)\neq\emptyset$ if and only if $|\sin 2\theta|<\exp(-2T)$, cf. Proposition \ref{main-estimate}. In particular, for $T$ large enough (depending on $\bar\theta_0$), we have that $|\theta|\leq\bar\theta_0/2$. It follows that, for any $|\theta'|\geq\bar\theta_0$, one has
$$\|g_t R_{\theta+\theta'} e_2\|^2=e^{2t}\sin^2(\theta+\theta')+e^{-2t}\cos^2(\theta+\theta')\geq e^{2t} \sin^2(\bar\theta_0/2)$$
Now, we observe that the proof of Proposition \ref{main-estimate} also gives that 
$$e^{2t}\geq \exp(-2T)\frac{1-\cos\omega}{2\sin^2\theta}$$
for any $t\in J(T,\theta)$, where $\sin2\theta=\exp(-2T)\sin\omega$. From this discussion, we obtain 
$$\|g_t R_{\theta+\theta'} e_2\|^2\geq \sin^2(\bar\theta_0/2)\exp(-2T)\frac{1-\cos\omega}{2\sin^2\theta}>\sin^2(\bar\theta_0/2)\frac{1-\cos\omega}{2\sin^2\theta}\|g_t R_{\theta} e_2\|^2$$
for $T$ large enough (depending on $\bar\theta_0$) and $t\in J(T,\theta)$. On the other hand, since $\sin2\theta=\exp(-2T)\sin\omega$, if $T$ is large than an absolute constant so that $|\cos\theta|\geq 1/2$, then 
$$2(1-\cos\omega)\geq 1-\cos^2\omega=\sin^2\omega=\exp(4T)\sin^22\theta\geq \exp(4T)\sin^2\theta,$$
that is, $(1-\cos\omega)/\sin^2\theta\geq \exp(4T)/2$. So, by putting these estimates together, we obtain that if $T$ is large enough (depending on $\bar\theta_0$ and $A$) and $|\theta'|\geq\bar\theta_0$, then 
$$\|g_t R_{\theta+\theta'} e_2\|^2\geq \sin^2(\bar\theta_0/2)(\exp(4T)/4)\|g_t R_{\theta} e_2\|^2\geq A^2\|g_t R_{\theta} e_2\|^2,$$
that is, \eqref{e.C} holds.
\bigskip

We claim that \eqref{e.A} holds for $\rho \leq \rho_0 \exp (-T_0)$. Indeed, writing $\rho = \rho_0 \exp (-T)$ with $T \geq T_0$:

\begin{itemize}
\item The measure of elements in $  \{ {\rm sys}(M) \leq \rho \}$ which are {\bf not} in  $Y(T,\frac{\pi}2 ,X_0^*)$ is $< \frac {\eta}2 \rho^2$. Indeed, from (\ref {e.F'}), (\ref {e.F}), Corollary \ref {measure-estimate} and Proposition \ref{Estimate-of-F}, we have
   
$$F(\rho) < \frac 12 (c(m)+ \frac {\eta}2) \rho^2$$
 and 
$$ m(Y(T,\frac{\pi}2 ,X_0^*)) = \frac 12 \exp(-2T) \rho_0 F'(\rho_0) > \frac 12 \rho^2 (c(m)- \frac {\eta}2).$$

\item For $M \in Y(T,\frac{\pi}2 ,X_0^*)$, write $M = g_t R_{\theta} M_0$, with $M_0 \in X_0^*$, $ |\sin 2\theta| < \exp (-2T) $, $t \in  J(T, \theta)$ . 
From (\ref {e.omega}) and Corollary \ref {measure-estimate}, we get 
$$ m(Y(T, \omega_0,X_0^*)) < \frac 14 \eta \rho^2.$$
Similarly, from (\ref {e.theta}) and Corollary \ref {measure-estimate}, we get 
$$ m(Y(T,\frac {\pi}2, B(\bar \theta_0))) < \frac 14 \eta \rho^2.$$
  
 \item  Assume that  $M = g_t R_{\theta} M_0$, with $\bar \theta(M_0) \geq \bar \theta_0$,  $ |\sin 2\theta| \geq \exp (-2T)  \sin \omega_0$. Assume also that $T \geq T_0$. Then the point $M$
 does not belong to  $\mathcal C(A,\rho )$. This follows from \eqref{e.B} for the saddle connections of $M_0$ of length $> AK \rho_0$, and from  \eqref{e.C} for the saddle connections of $M_0$ of length $\leq AK \rho_0$.
\end{itemize}

This concludes the proof of  \eqref{e.A}  and also of the theorem. 
\end{proof}

\end{document}